\documentclass{amsart}

\newtheorem{theorem}{Theorem}[section]

\newtheorem{corollary}[theorem]{Corollary}

\newtheorem{problem}[theorem]{Problem}

\theoremstyle{definition}

\newtheorem{example}[theorem]{Example}

\usepackage{amscd,amssymb}

\begin{document}

\title[Symmetric polynomials and their inner automorphisms]
{Symmetric polynomials in Leibniz algebras \\ and their inner automorphisms}

\author[{\c S}ehmus F{\i}nd{\i}k, Zeynep \"Ozkurt]
{{\c S}ehmus F{\i}nd{\i}k and Zeynep \"Ozkurt}

\address{Department of Mathematics,
\c{C}ukurova University, 01330 Balcal\i,
 Adana, Turkey}
\email{sfindik@cu.edu.tr}
\email{zyapti@cu.edu.tr}
\thanks
{}


\begin{abstract}

Let $L_n$ be the free metabelian Leibniz algebra generated by the set $X_n=\{x_1,\ldots,x_n\}$
over a field $K$ of characteristic zero. This is the free algebra of rank $n$ in the variety of solvable of class $2$ Leibniz algebras.
We call an element $s(X_n)\in L_n$ {\it symmetric} if $s(x_{\sigma(1)},\ldots,x_{\sigma(n)})=s(x_1,\ldots,x_n)$
for each permutation $\sigma$ of $\{1,\ldots,n\}$. The set $L_n^{S_n}$ of symmetric polynomials
of $L_n$ is the algebra of invariants of the symmetric group $S_n$. 
Let $K[X_n]$ be the usual polynomial algebra with indeterminates from $X_n$.
The description of the algebra $K[X_n]^{S_n}$ is well known, and 
the algebra $(L_n')^{S_n}$ in the commutator ideal $L_n'$ is a right $K[X_n]^{S_n}$-module.
We give explicit forms of elements of the $K[X_n]^{S_n}$-module $(L_n')^{S_n}$.
Additionally, we determine the description of the group $\text{\rm Inn}(L_{n}^{S_n})$
of inner automorphisms of the algebra $L_n^{S_n}$. The findings can be considered as a generalization of
the recent results obtained for the free metabelian Lie algebra of rank $n$.

\end{abstract}

\maketitle

\section*{Introduction}

Hilbert's fourteen problem is one of those famous twenty three problems suggested by
German mathematician David Hilbert \cite{H} in 1900
at the Paris conference of the International Congress of Mathematicians, and it is
related with the finite generation of the algebra $K[X_n]^G$ of invariants of $G<GL_n(K)$,
where $K[X_n]=K[x_1,\ldots,x_n]$ is the usual polynomial algebra over a field $K$, and $GL_n(K)$ is the general linear group.
Nagata \cite{Na} showed that the problem is not true in general in 1959. Earlier in 1916, Noether \cite{N} solved the problem in affirmative
for finite groups. In particular let $G=S_n$ be the symmetric group acting on the algebra $K[X_n]$ by permuting the variables:
$\pi\cdot p(x_1,\ldots,x_n)=p(x_{\pi(1)},\ldots,x_{\pi(n)})$, $p\in K[X_n]$, $\pi\in S_n$. The algebra $K[X_n]^{S_n}$
is generated by the set $\{\sum_{i=1}^n x_i^k\mid k=1,\ldots,n\}$, by the fundamental theorem of symmetric polynomials.
Elementary symmetric polynomials $e_j=\sum x_{i_1}\cdots x_{i_j}$, $i_1<\cdots<i_j$, $ j=1,\ldots,n$, form another generating set.

A noncommutative analogue of the problem is the algebra $K\langle X_n\rangle^{S_n}$ of symmetric polynomials
in the free associative algebra $K\langle X_n\rangle$. One may see the works \cite{Wo,BRRZ,GKL} on the algebra $K\langle X_n\rangle^{S_n}$.
Another analogue is working in relatively free Lie algebras, which are not associative and commutative.
The algebras $F_n^{S_n}$, and $M_n^{S_n}$ are not finitely generated via \cite{Br} and \cite{Dr},
where $F_n$ and $M_n$ are the free Lie algebra and the free metabelian Lie algebra of rank $n$, respectively.
One may see the papers \cite{FO,DFO} for the explicit elements of the algebra $M_n^{S_n}$. See also \cite{FO2} for the inner automorphisms 
of $M_n^{S_n}$. 

We consider the Leibniz algebras which can be thought as a generalization of the Lie algebras.
Leibniz algebras are defined by the identity $[x,[y,z]]=[[x,y],z]-[[x,z],y]$, where the bracket
is bilinear; however, non-necessarily skew-symmetric. In the case of skew-symmetry the identity turns into the Jacobi identity, and we obtain a Lie algebra.
Leibniz algebras are related with many branches of mathematics. See the papers \cite{L,LP,MU,AO,O} for more details.

In the present study, we consider the free metabelian Leibniz algebra $L_n$ and we determine the algebra $L_n^{S_n}$ of symmetric polynomials.
Additionally, we describe the group $\text{\rm Inn}(L_n^{S_n})$ of inner automorphisms of $L_n^{S_n}$.

\section{Preliminaries}

Let $K$ be a field of characteristic zero. A Leibniz algebra $L$ over $K$
is a vector space furnished with bilinear commutator $[.,.]$ satisfying the
Leibniz identity
\[
[[x,y],z]=[[x,z],y]+[x,[y,z]],
\]
or
\[
[x,y]r_z=[xr_z,y]+[x,yr_z],
\]
$x,y,z\in L$. Here $r_z$ stands for the adjoint operator $\text{\rm ad}z$ acting from right side by commutator multiplication.
The Leibniz algebra $L$ is nonassociative and noncommutative. 
Leibniz algebras are generalizations of Lie algebras.
There is a natural way to obtain a Lie algebra from a given Leibniz algebra $L$.
Consider the ideal $\text{\rm Ann}(L)$, called the annihilator of $L$, generated by all elements $[x,x]$, $x\in L$.
It is well known, see e.g. \cite{LP}, that $r_a=0$ if and only if $a\in \text{\rm Ann}(L)$.
Clearly the elements
of the form $[x,y]+[y,x]\in\text{\rm Ann}(L)$ by the equality
\[
[x,y]+[y,x]=[x+y,x+y]-[x,x]-[y,y], \ \ x,y\in L.
\]
Hence the factor algebra $L/\text{\rm Ann}(L)$ turns into a Lie algebra,
since the anticommutativity $[x,y]+[y,x]=0$ and the Jacobi identity
\[
[[x,y],z]+[[y,z],x]+[[z,x],y]=0,
\]
are satisfied in $L/\text{\rm Ann}(L)$.

\

Now consider the free algebra $L_n$ of rank $n$ generated by $X_n=\{x_1,\ldots,x_n\}$ in the variety of metabelian Leibniz algebras over the base field $K$.
The algebra $L_n$ satisfies the metabelian identity $[[x,y],[z,t]]=0$, and is a solvable of class $2$ Leibniz algebra.
Hence every element in the commutator ideal $L_n'=[L_n,L_n]$
of the free metabelian Leibniz algebra $L_n$ can be expressed as a linear combination of left-normed monomials of the form
\begin{align}
[[\cdots[[x_{i_1},x_{i_2}],x_{i_3}],\ldots],x_{i_k}]&=[x_{i_1},x_{i_2},x_{i_3},\ldots,x_{i_k}]\nonumber\\
&=[x_{i_1},x_{i_2}]r_{i_3}\cdots r_{i_k}=[x_{i_1},x_{i_2}]r_{i_{\pi(3)}}\cdots r_{i_{\pi(k)}}\nonumber
\end{align}
where $\pi$ is a permutation of the set $\{3,\ldots,k\}$. In this way the commutator ideal $L_n'$ can be considered as a right $K[R_n]=K[r_1,\ldots,r_n]$-module,
where $r_i=r_{x_i}=\text{\rm ad}x_i$, $i=1,\ldots,n$. It is well known, see e.g. Proposition 3.1. of the paper \cite{DP}, that the elements 
\[
x_{i_1}, [x_{i_1},x_{i_2}], [x_{i_1},x_{i_2}]r_{i_3}\cdots r_{i_k}, \ \  1\leq i_1, i_2\leq n, \ \  1\leq i_3\leq\cdots\leq i_k,
\]
form a basis for $L_n$. The next result is a direct consequence of this basis.

\begin{corollary}\label{free}
The commutator ideal $L_n'$ of the free metabelian Leibniz algebra $L_n$ is a free right $K[R_n]$-module with generators
$[x_i,x_j]$, $1\leq i, j\leq n$.
\end{corollary}

A polynomial $s=s(x_1,\ldots,x_n)$ in the free metabelian Leibniz algebra $L_n$ is said to be {\it symmetric} if
\[
\sigma s=s(x_{\sigma(1)},\ldots,x_{\sigma(n)})=s(x_1,\ldots,x_n), \ \ \sigma\in S_n.
\]
The set $L_n^{S_n}$ of symmetric polynomials forms a Leibniz subalgebra,
which is the algebra of invariants of the symmetric group $S_n$. The $K[R_n]$-module structure of the commutator ideal $L_n'$
implies that the algebra $(L_n')^{S_n}$ is a right $K[R_n]^{S_n}$-module. One of the set of generators of the algebra $K[R_n]^{S_n}$
of symmetric polynomials is well known: $\{r_1^k+\cdots+r_n^k\mid 1\leq k\leq n\}$, see e.g. \cite{S}.

\section{Main Results}

\subsection{Symmetric polynomials}

In this section we determine the algebra $L_n^{S_n}$ of symmetric polynomials in the free metabelian Leibniz algebra $L_n$.
Clearly the linear symmetric polynomials are included in the $K$-vector space spanned on a single element $x_1+\cdots+x_n$.
Hence it is sufficient to work in the commutator ideal $L_n'$ of $L_n$, and describe the algebra $(L_n')^{S_n}$.
Let us fix the notations $a_i=[x_i,x_i]$, $1\leq i\leq n$, and $b_{jk}=[x_j,x_k]$, $1\leq j\neq k\leq n$, 
which are the free generators of $K[R_n]$-module $L'$.
We provide explicit elements of $K[R_n]^{S_n}$-module $(L_n')^{S_n}$. For this purpose, we study in the
$K[R_n]$-submodules
\[
A_n=\left\{\sum_{i=1}^na_ip_i\mid p_i\in K[R_n]\right\} \ \ \text{\rm and} \ \ B_n=\left\{\sum_{1\leq j\neq k\leq n}b_{jk}q_{jk}\mid q_{jk}\in K[R_n]\right\}
\]
of the $K[R_n]$-module $L_n'=A_n\oplus B_n$, generated by $a_i$, $1\leq i\leq n$, and $b_{jk}$, respectively, $1\leq j\neq k\leq n$,
due to the fact that they are invariant under the action of $S_n$; i.e., $A_n^{S_n}\subset A_n$, and $B_n^{S_n}\subset B_n$.

Let us denote the subgroups $\Pi_i=\{\pi\in S_n\mid \pi(i)=i\}$, $1\leq i\leq n$, and $\Pi_{jk}=\{\pi\in S_n\mid \pi(j)=j, \ \pi(k)=k\}$, $1\leq j\neq k\leq n$, of $S_n$.
In the next theorems, we determine symmetric polynomials in the $K[R_n]$-modules $A_n$, and $B_n$, respectively.

\begin{theorem}\label{ii}
Let $p=\sum_{i=1}^na_{i}p_i$ be a polynomial in $A_n$, for some $p_i\in K[R_n]$, $1\leq i\leq n$.
Then $p$  is symmetric if and only if
\[
p_1(r_1,r_2,\ldots,r_n)=\pi p_1(r_1,r_2,\ldots,r_n)=p_1(r_1,r_{\pi(2)},\ldots,r_{\pi(n)}), \ \ \pi\in\Pi_1,
\]
$\sigma p_i=p_i$, $\sigma\in\Pi_i$, and $p_i=(1i)p_1=p_1(r_i,r_2,\ldots,r_{i-1},r_1,r_{i+1},\ldots,r_n)$, for transpositions $(1i)\in S_n$, $i=2,\ldots,n$.
\end{theorem}

\begin{proof}
Let $p\in A_n$ be an element of the form
\[
p=\sum_{i=1}^na_{i}p_i(r_1,\ldots,r_n), \ \ p_i\in K[R_n].
\]
If $p$ is a symmetric polynomial, then $p=\pi p$; i.e.,
\[
\sum_{i=1}^na_{i}p_i(r_1,\ldots,r_n)=\sum_{i=1}^na_{\pi(i)}p_i(r_{\pi(1)},\ldots,r_{\pi(n)})
\]
for each $\pi\in S_n$ by definition, and by Corollary \ref{free} we may compare the coefficients of $a_i$, $i=1,\ldots,n$,
from $K[R_n]$, in the last equality.
In particular, $p=\pi p$ for each $\pi\in\Pi_i$, and comparing the coefficients of $a_i$, $i=1,\ldots,n$, we obtain that
\[
p_i(r_1,\ldots,r_n)=p_i(r_{\pi(1)},\ldots,r_{\pi(i-1)},r_i,r_{\pi(i+1)},\ldots,r_{\pi(n)}).
\]
Now consider $(1i)p=p$ for every transposition $(1i)\in S_n$, $i=2,\ldots,n$. Then these equalities give that
$p_i=(1i)p_1$, and thus
\begin{align}
p_i(r_1,\ldots,r_n)&=(1i)p_1(r_1,\ldots,r_n)\nonumber\\
&=p_1(r_i,r_2,\ldots,r_{i-1},r_1,r_{i+1},\ldots,r_n)\nonumber
\end{align}
Conversely consider the element $p=\sum_{i=1}^na_ip_i$ satisfying the conditions in the theorem.
It is sufficient to show that $(1k)p=p$, $k=2,\ldots,n$, since these transpositions generate the symmetric group $S_n$.
Note that if $i\neq 1,k$, then $(1k)p_i=p_i$, since $(1k)\in\Pi_i$.
The following computations complete the proof.
\begin{align}
(1k)p=&(1k)\left(a_1p_1+a_kp_k+\sum_{i\neq 1,k}a_ip_i\right)\nonumber\\
&=a_k((1k)p_1)+a_1((1k)p_k)+\sum_{i\neq 1,k}a_i((1k)p_i)\nonumber\\
&=a_kp_k+a_1p_1+\sum_{i\neq 1,k}a_ip_i=p.\nonumber
\end{align}
\end{proof}

\begin{theorem}\label{ij}
Let $q=\sum b_{ij}q_{ij}$ be a polynomial in $B_n$, for some $q_{ij}\in K[R_n]$, $1\leq i\neq j\leq n$.
Then $q$  is symmetric if and only if $q_{ij}=\sigma q_{kl}$ for every $\sigma:i\to k,j\to l$, in particular,
\begin{align}
q_{1i}=&(2i)q_{12},\ \ q_{i2}=(1i)q_{12}, \ \ q_{2i}=(1i)q_{21}, \ \ q_{i1}=(2i)q_{21}, \nonumber\\
&q_{21}=(12)q_{12}, \ \ q_{ij}=(1i)(2j)q_{12}, \ \ 3\leq i\neq j\leq n.\nonumber
\end{align}
and $q_{ij}=\pi q_{ij}$, for all $\pi\in\Pi_{ij}$. 
\end{theorem}

\begin{proof}
Assume that a polynomial $q=\sum b_{ij}q_{ij}\in B_n$, $1\leq i\neq j\leq n$,
is symmetric. Then $\pi q=q$ for each $\pi\in\Pi_{12}$ gives that
\[
q_{12}(r_1,r_2,\ldots,r_n)=q_{12}(r_1,r_2,r_{\pi(3)},\ldots,r_{\pi(n)}).
\]
Relations on $q_{ij}$'s in the theorem are straightforward, by making use of Corollary \ref{free},
and comparing the coefficients of $b_{12}$, $b_{21}$, and $b_{ij}$ from the equalities
$q=(12)q=(1i)q=(2i)q=(1i)(2j)q=(ij)q$, where $3\leq i\neq j\leq n$.

Now let the polynomial $q\in B_n$ satisfy the conditions of the theorem, and $(1k)\in S_n$ be a transposition
for a fixed $k\in\{3,\dots,n\}$. We have to show that $(1k)q=q$.
Let express $q$ in the following form.
\[
q=b_{1k}q_{1k}+b_{k1}q_{k1}+\sum_{i,j\neq1,k}b_{ij}q_{ij}+\sum_{i\neq 1,k}(b_{1i}q_{1i}+b_{ki}q_{ki}+b_{ik}q_{ik}+b_{i1}q_{i1}).
\]
Note that $(1k)\in\Pi_{ij}$, and hence $(1k)q_{ij}=q_{ij}$, for $i,j\neq1,k$. Then we have that
\begin{align}
(1k)q=&b_{k1}((1k)q_{1k})+b_{1k}((1k)q_{k1})+\sum_{i,j\neq1,k}b_{ij}((1k)q_{ij})\nonumber\\
&+\sum_{i\neq 1,k}(b_{ki}((1k)q_{1i})+b_{1i}((1k)q_{ki})+b_{i1}((1k)q_{ik})+b_{ik}((1k)q_{i1}))\nonumber\\
=&b_{k1}q_{k1}+b_{1k}q_{1k}+\sum_{i,j\neq1,k}b_{ij}q_{ij}+\sum_{i\neq 1,k}(b_{ki}q_{ki}+b_{1i}q_{1i}+b_{i1}q_{i1}+b_{ik}q_{ik})=q.\nonumber
\end{align}
\end{proof}

We obtain the next corollary by combining Theorem \ref{ii} and Theorem \ref{ij}.

\begin{corollary}
If $s$ is a symmetric polynomial in the free metabelian Leibniz algebra $L_n$, then it is of the form
\begin{align}
s=&\sum_{1\leq i\leq n}\alpha x_i+\sum_{1\leq i\leq n}[x_i,x_i]((1i)f)+\sum_{3\leq i\neq j\leq n}[x_i,x_j]((1i)(2j)g)\nonumber\\
&+[x_1,x_2]g+\sum_{3\leq i\leq n}([x_1,x_i]((2i)g)+[x_i,x_2]((1i)g)\nonumber\\
&+[x_2,x_1]h+\sum_{3\leq i\leq n}([x_i,x_1]((2i)h)+[x_2,x_i]((1i)h)),\nonumber
\end{align}
where $\alpha\in K$, $f,g,h\in K[R_n]$, such that $\pi f=f$ for $\pi\in\Pi_{1}$, $\sigma g=g$ for $\sigma\in\Pi_{12}$, and $h=(12)g$.
\end{corollary}

\begin{example}
Let $n=2$ and the free metabelian Leibniz algebra $L_2$ be generated by $x_1,x_2$.
Then each symmetric polynomial $s\in L_2^{S_2}$ is of the form 
\begin{align}
s=&\alpha(x_1+x_2)+[x_1,x_1]f(r_1,r_2)+[x_2,x_2]f(r_2,r_1)\nonumber\\
&+[x_1,x_2]g(r_1,r_2)+[x_2,x_1]g(r_2,r_1),\nonumber
\end{align}
where $\alpha\in K$, $f,g\in K[R_2]$. Note that the Lie correspondence of this result (modulo the annihilator)
is that if $s(x_1,x_2)$ is a symmetric polynomial in the free metabelian Lie algebra generated by $x_1,x_2$, then
\[
s=\alpha(x_1+x_2)+[x_1,x_2]t(r_1,r_2),
\]
such that $t(r_1,r_2)=-t(r_2,r_1)$, which is compatible with the recent result given in \cite{FO}.
\end{example}

\subsection{Inner automorphisms}

Let $u$ be an element in the commutator ideal $L_n'$ of the free metabelian Leibniz algebra $L_n$. The adjoint operator
\[
\text{\rm ad}u:L_n\to L_n, \ \ \text{\rm ad}u(v)=[v,u], \ \ v\in L_n
\]
is nilpotent since $\text{\rm ad}^2u=0$, and that $\psi_u=\exp(\text{\rm ad}u)=1+\text{\rm ad}u$
is called an inner automorphism of $L_n$ with inverse $\psi_{-u}$. Clearly the group $\text{\rm Inn}(L_n)$
consisting of all inner automorphisms is abelian due to the fact that $\psi_{u_1}\psi_{u_2}=\psi_{u_1+u_2}$.

In the next theorem we determine the group $\text{\rm Inn}(L_n^{S_n})$ of inner automorphisms preserving symmetric polynomials.

\begin{theorem}
$\text{\rm Inn}(L_n^{S_n})=\{\psi_{u_1+u_2}\mid  u_1\in \text{\rm Ann}(L_n),  \ u_2\in (L_n')^{S_n}\}$.
\end{theorem}

\begin{proof}
Let $v\in L_n^{S_n}$, $u=u_1+u_2$, for some $u_1\in \text{\rm Ann}(L_n)$, and $u_2\in (L_n')^{S_n}$. Then clearly
\[
\psi_u(v)=v+[v,u_1+u_2]=v+[v,u_2]\in L_n^{S_n}.
\]
Conversely, let $\psi_u(v)\in L_n^{S_n}$ for $v\in L_n^{S_n}$, and $u\in L_n'$. The action of $\psi_u$ is identical when $v\in L_n'$.
Hence we assume that the linear counterpart $v_l=\alpha(x_1+\cdots+x_n)$, $\alpha\in K$, of $v$ is nonzero.
We may express $u=u_1+u_2$, $u_1\in \text{\rm Ann}(L_n)$, $u_2\in L_n'$, where $u_2\notin \text{\rm Ann}(L_n)$.
Hence we have $\psi_u(v)\in L_n^{S_n}$, which implies that $[v_l,u_2]$ is a symmetric polynomial. Let $\pi\in S_n$
be an arbitrary permutation. Then
\[
[v_l,u_2]=\pi[v_l,u_2]=[\pi v_l,\pi u_2]=[v_l,\pi u_2]
\]
or $[x_1+\cdots+x_n,u_2-\pi u_2]=0$, and thus $u_2-\pi u_2=0$. Therefore $u_2\in (L_n')^{S_n}$.
\end{proof}

\

We complete the paper by releasing the following problem.

\begin{problem}
Determine the group $\text{\rm Aut}(L_n^{S_n})$ of all automorphisms preserving the symmetric polynomials.
\end{problem}

\end{document}